\documentclass{amsart}
\usepackage{amssymb}
\usepackage{amscd}
\usepackage{amsfonts}
\usepackage{latexsym}
\usepackage{amsmath}
\usepackage{mathbbol}
\usepackage[all]{xy}
\usepackage{verbatim}
\usepackage{tikz-cd}

\newtheorem{theorem}{Theorem}[section]
\newtheorem{lemma}[theorem]{Lemma}
\newtheorem{proposition}[theorem]{Proposition}
\newtheorem{corollary}[theorem]{Corollary} 
\theoremstyle{definition}  
\newtheorem{definition}[theorem]{Definition}
\newtheorem{example}[theorem]{Example}
\newtheorem{conjecture}[theorem]{Conjecture}  

\newtheorem{remark}[theorem]{Remark}

\newcommand{\rank}{\text{rank}}

\newcommand{\FPdim}{\text{FPdim}} 

\renewcommand{\Vec}{\text{Vec}}

\newcommand{\BrPic}{\text{BrPic}}

\newcommand{\Hom}{\text{Hom}} 
\newcommand{\uHom}{\underline{\text{Hom}}}

\newcommand{\Rep}{\text{Rep}}

\newcommand{\rev}{\text{rev}}

\newcommand{\sVec}{\text{s}\Vec}

\newcommand{\B}{\mathcal{B}}
\newcommand{\C}{\mathcal{C}}
\newcommand{\D}{\mathcal{D}}
\newcommand{\E}{\mathcal{E}}

\newcommand{\Z}{\mathcal{Z}}

\renewcommand{\L}{\mathcal{L}}
\newcommand{\M}{\mathcal{M}}

\newcommand{\N}{\mathcal{N}}

\newcommand{\be}{\mathbf{1}}

\renewcommand{\be}{\mathbf{1}}

\newcommand{\mN}{{\mathcal N}}
\newcommand{\cC}{{\mathcal C}}

\newcommand{\mD}{{\mathcal D}}

\newcommand{\mM}{{\mathcal M}}
\newcommand{\bt}{\boxtimes}

\newcommand{\ot}{\otimes}

\newcommand{\bth}{\begin{theorem}}
\renewcommand{\eth}{\end{theorem}}
\newcommand{\bpr}{\begin{proposition}}
\newcommand{\epr}{\end{proposition}}
\newcommand{\ble}{\begin{lemma}}
\newcommand{\ele}{\end{lemma}}
\newcommand{\bco}{\begin{corollary}}
\newcommand{\eco}{\end{corollary}}
\newcommand{\bde}{\begin{definition}}
\newcommand{\ede}{\end{definition}}
\newcommand{\bex}{\begin{example}}
\newcommand{\eeqx}{\end{example}}
\newcommand{\bre}{\begin{remark}}
\newcommand{\ere}{\end{remark}}
\newcommand{\bcj}{\begin{conjecture}}
\newcommand{\ecj}{\end{conjecture}}

\newcommand{\beq}{\begin{equation}}
\newcommand{\eeq}{\end{equation}}

\newcommand{\bpf}{\begin{proof}}
\newcommand{\epf}{\end{proof}}

\begin{document}
\title{Rank-finiteness for G-crossed braided fusion categories}
\author{Corey Jones, Scott Morrison, Dmitri Nikshych, Eric C. Rowell}
\date{\today}
\thanks{ECR is partially supported by NSF grant DMS-1664359.  
DN was partially supported  by the NSA
grant H98230-16-1-0008 and the NSF grant DMS-1801198. 
This paper was initiated while ECR and DN were visiting CJ and SM at the Australian National University, and gratefully acknowledge the support of that institution.}
\begin{abstract} We establish rank-finiteness for the class of $G$-crossed braided fusion categories, generalizing the recent result for modular categories and including the important case of braided fusion categories. This necessitates a study of slightly degenerate braided fusion categories and their centers, which are interesting for their own sake. 
\end{abstract} 
\maketitle  

\section{Introduction}

The question of whether there are finitely many fusion categories with a fixed number of isomorphism classes of simple objects (i.e., fixed \emph{rank}) was first raised by Ostrik in \cite{OstrikRank2}, where an affirmative answer was given for rank $2$.  In \cite{ENO1} the special case of categories with integral Frobenius-Perron dimension (i.e. \emph{weakly integral} categories) was also settled.  Around 2003 Wang conjectured that there are always finitely many \emph{modular} categories of a given fixed rank, which was explicitly verified for rank at most $4$.  A proof of this rank-finiteness conjecture was obtained recently \cite{BNRZ}.  The main goal of this article is to extend rank-finiteness to the generality of $G$-crossed braided fusion categories, which includes the important case of braided fusion categories, and does not require the existence of a spherical structure.

The primary obstacle to overcome is the existence of \emph{slightly degenerate} braided fusion categories, with symmetric center equivalent the braided fusion category $\sVec$ of super vector spaces.  These are interesting in their own right, with the main open question being whether or not every slightly degenerate braided fusion category admits a minimal non-degenerate extension.  As a step towards answering this question we analyze the structure of the Drinfeld  center
of a slightly degenerate braided fusion category.

As a technical tool, we prove a bound on the rank of invertible $(\C-\D)$-bimodule categories. In particular, we show that for any invertible 
$\C$-bimodule category, $\rank(\M)\le \rank(\C)$. In addition, we show that the set of equivalence classes of invertible bimodule categories realizing this bound forms a subgroup of $\operatorname{BrPic}(\C)$, and discuss some examples. 
\section{Preliminaries}

We work over an algebraically closed field $k$ of characteristic $0$. All fusion categories and their module categories
are assumed to be $k$-linear. For the basics of the theory
of fusion categories we refer the reader to \cite{EGNO} and \cite{DGNO}. 

By the {\em rank} of a  fusion category we mean the number of isomorphism classes of its simple objects.

Let $\Vec$ and $\sVec$ denote the braided fusion categories of vector spaces and super vector spaces over $k$.
For any braided fusion category $\C$ let $\Z_{sym}(\C)$ denote its symmetric (or M\"uger) center.

\begin{definition}
A braided fusion category $\C$ is called {\em slightly degenerate} \cite{DNO} if $\Z_{sym}(\C)=\sVec$.
A slightly degenerate ribbon fusion category is called {\em super-modular}.
\end{definition}

The smallest example of a slightly degenerate braided fusion category is $\sVec$
itself.

\begin{example}
\label{a typical example}
One can construct a slightly degenerate braided fusion category as follows. 
Let $\tilde{\C}$ be  a non-degenerate braided fusion category and let $\sVec \hookrightarrow \tilde{\C}$
be a braided tensor functor (it is automatically an embedding). 
Then the centralizer of the image of $\sVec$ in $\C$ is slightly degenerate. 

\end{example}

Let $\C$ be a slightly degenerate braided fusion category.
Below we recall  some facts  about $\C$ from \cite{DNO, BNRZ}. 

Let $\delta$ denote the simple object generating $\Z_{sym}(\C)$. 
Then $\delta \ot X \ncong X$ for each simple object $X$ in $\C$ (see \cite[Lemma 5.4]{Mu1} and \cite[Lemma 3.28]{DGNO}).  
In particular, the rank of a slightly degenerate braided fusion category is even.

We say that $\C$ is {\em split} if 
$\C \cong  \C_0 \bt \sVec$, where $\C_0$ is a non-degenerate braided fusion category.
Any pointed slightly degenerate braided fusion category is split, see \cite[Proposition 2.6(ii)] {ENO3} or \cite[Corollary A.19]{DGNO}.

The following definition is due to  M\"uger \cite{Mu2}.

\begin{definition}
\label{def min ext}
A {\em minimal extension} of a slightly degenerate braided fusion 
(respectively, super-modular) category  $\C$ is a braided tensor functor  $\iota:  \C \hookrightarrow \tilde{\C}$,
where $\tilde{\C}$ is a non-degenerate braided fusion 
(respectively, modular) category  such that  the centralizer of $\C$ in $\tilde{\C}$ is the image of $\sVec$.
\end{definition}

Note that the above functor $\iota$ is an embedding by \cite[Corollary 3.26]{DMNO}.

Clearly, every slightly degenerate braided fusion category that admits a minimal extension can obtained
via the construction from Example~\ref{a typical example} and vice versa.

An equivalence of minimal extensions is defined in an obvious way.

\begin{example}
The category $\sVec$ has $16$ inequivalent minimal extensions \cite{DNO,kitaev}: $8$ Ising categories and $8$ pointed categories.
The Witt classes of these extensions form a subgroup of the categorical Witt group isomorphic to $\mathbb{Z}/16\mathbb{Z}$. 
\end{example}

It follows that $\FPdim(\tilde{\C}) = 2 \FPdim(\C)$. By \cite{Mu1, DGNO} this is the {\em minimal} possible  value of
the Frobenius-Perron dimension of a non-degenerate braided fusion category containing $\C$. This explains our terminology.

\begin{lemma}
\label{Z2 grad}
Let $\D$ be a fusion category and let $\D_0 \subset \D$ be a fusion subcategory such that 
$\FPdim(\D) = 2 \FPdim(\D_0)$. Then  $\D$  is faithfully $\mathbb{Z}/2\mathbb{Z}$-graded with the trivial
component $\D_0$.
\end{lemma}
\begin{proof}
Let $\D= \D_0\oplus \D_1$ be a decomposition of $\D$ into the sum of $\D_0$ and its direct complement $\D_1$.
Then $\D_1$ is a $\D_0$-bimodule subcategory of $\D$. To prove the statement it suffices to check that the tensor 
product  of $\D$ maps $\D_1\times \D_1$ to $\D_0$. Let  $d:=\FPdim(\D_0)=\FPdim(\D_1)$.  
Let $R$ denote the (virtual) regular object in $\C$. We can write it as $R=R_0+R_1$, where $R_0$
is the regular object of $\D_0$ and $R_1$ is a regular object of the $\D_0$-module category $\D_1$ \cite{ENO1}
such that $\FPdim(R_0)=\FPdim(R_1)=d$.  Then $R_1R= dR =d(R_0+R_1)$. On the other hand,
\[
R_1R = R_1 (R_0+R_1)= R_1R_0 + R_1^2 = dR_1 + R_1^2,
\]
since a regular object of $\D_1$ is unique up to a scalar multiple.
Hence, $R_1^2=dR_0$, which implies that the tensor product of any two objects of $\D_1$ is in $\D_0$.
\end{proof}

Thus, a minimal extension of a slightly degenerate braided fusion category is the same thing as a faithful $\mathbb{Z}/2\mathbb{Z}$-extension which is a non-degenerate braided 
fusion category.


\section{Maximal rank bimodule categories}

In this section, we show that invertible bimodule categories over a fusion category exhibit a rank bound, and that the bimodule categories realizing this bound actually form a subgroup of the Brauer-Picard group. We refer the reader to \cite{ENO2} for definitions and properties of invertible bimodule categories.

\begin{proposition}
\label{rank bound}
Let $\C,\D$ be fusion categories, and $\M$ an invertible $(\C-\D)$-bimodule category. Then $\rank(\M)\le (\rank(\C)\rank(\D))^{\frac{1}{2}}$. In particular, for an invertible $\C-\C$ bimodule category, $\rank(\M)\le \rank(\C)$.
\end{proposition}

\begin{proof}

First consider $\M$ as a left $\C$-module category. Then the associated full center provides us with a Lagrangian algebra $L\in \Z(\C)$ \cite{DavydovCentre}. Let $F_{\C}:\Z(\C)\rightarrow \C$ be the forgetful functor, and $I_{\C}$ its adjoint. Then as an algebra in $\C$, $F_{\C}(L)\cong \bigoplus_{M\in \text{Irr}(\M)} \uHom(M,M),$ where the internal hom is taken as a left $\C$ module category. Note that each $\uHom(M,M)$ is a separable, connected algebra, and thus $\dim(\Hom_{\C}(\mathbb{1}, F_{\C}(L))=\rank(\M)$. But we have a canonical isomorphism 
\[
\Hom_{\C}(\mathbb{1}, F_{\C}(L))\cong \Hom_{\Z(\C)}(I_{\C}(\mathbb{1}), L).
\]

However, by \cite{ENO2}, the bimodule category $M$ induces a canonical braided equivalence $\alpha: \mathcal{Z}(\C)\rightarrow \mathcal{Z}(\D)$ such that $\alpha(L)\cong I_{\D}(\mathbb{1})$, thus we have 
\begin{eqnarray*}
 \dim(\operatorname{End}_{\Z(\C)}(I_{\C}(\mathbb{1})))
 &=&\dim(\Hom_{\C}(\mathbb{1}, F_{\C}(I_{\C}(\mathbb{1})) ) )=\rank(\C),\\
 \dim(\operatorname{End}_{\Z(\C)}(L)) &=&\dim(\operatorname{End}_{\Z(\D)}(I_{\D}(\mathbb{1})))=\rank(\D).
 \end{eqnarray*}

Here we have used that as an object $F_{\cC}(I(\mathbb{1}))\cong \bigoplus_{X\in \text{Irr}(\C)} X\otimes X^{*}.$
Therefore by the Cauchy-Schwartz inequality, 
\begin{eqnarray*}
\rank(\M)&=&\dim(\Hom_{\Z(\C)}(I(\mathbb{1}),L)) \\
&=& \sum_{X\in \text{Irr}(\Z(\C))} \dim(\Hom_{\Z(\C)}(I(\mathbb{1}), X)) \dim(\Hom_{\Z(\C)}(L, X))\\
&\le& \dim(\operatorname{End}_{\Z(\C)}(I(\mathbb{1})))^{\frac{1}{2}} \dim(\operatorname{End}_{\Z(\C)}(L)) ^{\frac{1}{2}} \\
&=&(\rank(\C)\rank(\D))^{\frac{1}{2}}.
\end{eqnarray*}
\end{proof}

\begin{remark}
Note the bound $\rank(\M)\le \rank(\C)$ requires invertibility. Consider for example the rank $4$ fusion category $\C=\operatorname{Rep}(D_{5})$,
where $D_5$ is the group of symmetries of the regular pentagon. Then there exists a rank $5$ indecomposable bimodule category, namely $\operatorname{Rep}(\mathbb{Z}_5)$, where the (left and right) actions of $\operatorname{Rep}(D_{5})$ are induced from the restriction functor
(here $\mathbb{Z}_5$ is the subgroup of rotations of $D_5$).
\end{remark}

The above proposition leads us to the following definition.

\begin{definition} We say that an invertible $\C$-bimodule category $\M$ has \textit{maximal rank} if $\rank(\M)=\rank(\C)$.
\end{definition}

\begin{proposition}\label{maxranksubgroup}
Let $\Psi:\operatorname{BrPic}(\C)\rightarrow \operatorname{Aut}_{br}(\mathcal{Z}(\C))$ be the canonical group isomorphism of \cite{ENO2}. Then $\M$ is maximal rank if and only if $\Psi(\M)$ preserves the isomorphism class of the object $I(\mathbb{1})$.
\end{proposition}

\begin{proof}

Returning to the proof of Proposition \ref{rank bound} and identifying $\D$ with $\C$ then $\Psi(\M)=\alpha$, and we are interested in the case when the Cauchy-Schwartz inequality yields equality. But this happens precisely when there exists a scalar $\lambda$ such that
\[
\dim(\Hom_{\Z(\C)}(I(\mathbb{1}), X))=\lambda \dim(\Hom_{\Z(\C)}(\alpha(I(\mathbb{1})), X)).
\]

But 
\begin{eqnarray*}
\rank(\C)&=& \sum_{X\in \text{Irr}(\mathcal{Z}(\C)}\dim(\Hom_{\Z(\C)}(I(\mathbb{1}), X))^{2}\\
&=&\lambda^{2}\sum_{X\in \text{Irr}(\mathcal{Z}(\C)}\dim(\Hom_{\Z(\C)}(\alpha(I(\mathbb{1})), X))^{2}=\lambda^{2}\,\rank(\C).
\end{eqnarray*}

Since the dimension of morphism spaces is non-negative, we see that we must have $\lambda=1$. Thus
\[
\dim(\Hom_{\Z(\C)}(I(\mathbb{1}), X))=\dim(\Hom_{\Z(\C)}(\alpha(I(\mathbb{1})), X))
\]
for all $X\in \text{Irr}(\mathcal{Z}(\C))$ and the conclusion follows.

\end{proof}

\begin{corollary}
The maximal rank invertible bimodule categories form a subgroup of $\operatorname{BrPic}(\C)$.
\end{corollary}

This result seems somewhat surprising, since in general the behavior of the rank of bimodule categories is notoriously difficult to understand under relative tensor products.

Recall there is a canonical subgroup $\operatorname{Out}(\C)\le \operatorname{BrPic}(\C)$ which consists of equivalence classes of invertible bimodule categories which are trivial as a left module category. This implies the right action must be the usual right action twisted by an auto-equivalence of $\C$. More explicitly, let $\beta$ be a tensor autoequivalence of $\C$ and $\C_{\beta}$ the associated bimodule category, which is $\C$ as an underlying category and with actions $X\triangleright Y=X\otimes Y$, $X\triangleleft Y=X\otimes \beta(Y)$, and the obvious associators. The image of these bimodule categories in $\operatorname{BrPic}(\C)$ forms the subgroup $\operatorname{Out}(\C)$.

Using the correspondence between module categories and Lagrangian algebras, we see that this is precisely the subgroup of $\operatorname{BrPic}(\C)$ which preserve $I(\mathbb{1})$ \textit{as an algebra object}. In particular, $\operatorname{Out}(\C)$ forms a subgroup of the maximal rank bimodule categories. In many cases, this is the whole group.

\begin{proposition}\label{pointedfuse}
For any pointed fusion category $\C$, the group of maximal rank bimodule categories is $\operatorname{Out}(\C)$.
\end{proposition}

\begin{proof}
Any pointed fusion category $\C$ is monoidally equivalent to $\operatorname{Vec}(G,\omega)$ for a finite group $G$ and $3$-cocycle $\omega\in \text{Z}^{3}(G, \mathbb{C}^{\times})$. By \cite{OstrikModule}, the module categories for this fusion category are classified by subgroups $H\le G$ together with a trivialization of $\omega|_{H}$. The rank of the resulting module category is the index $[G:H]$. Thus there is a unique rank $|G|$ indecomposable module category, where $H=\{e\}$, which is $\operatorname{Vec}(G,\omega)$ acting on itself. The dual category is thus $\operatorname{Vec}(G,\omega)$, hence any invertible rank $|G|$ bimodule category is of the form $\operatorname{Out}(\C)$.
\end{proof}

There exist maximal rank invertible bimodule categories that are not of the form $\operatorname{Out}(\C)$. One such example is constructed by Ostrik in the appendix of \cite{CMS} using an extension of the Izumi-Xu fusion category.
See \cite[Theorem A.5.1]{CMS} and \cite[Remark 2.19 and Example 2.20]{OstrikPivotal}.

To find a maximal rank bimodule category not of the form $\operatorname{Out}(\C)$, we need not only a distinct etale algebra structure on $I(\mathbb{1})$, but we need this algebra structure to be the image of $I(\mathbb{1})$ under a braided autoequivalence, which makes finding invertible bimodule categories not of the form $\operatorname{Out}(\C)$ difficult in general. 

To find such examples, we move in a different direction. If $\C$ is braided, we can try to understand invertible module categories over $\C$. Recall from \cite[Remark 2.13]{DN1}  that we can characterize the bimodule categories $\M\in \BrPic(\C)$ which are in the image of the map from $\operatorname{Pic}(\C)$ as the one-sided bimodule categories. By definition, these are bimodule categories for which there exists natural isomorphisms $d_{M,X}: M\triangleleft X\cong X\triangleright M$ satisfying a collection of coherences. It is not hard to see that these coherences imply the only one-sided invertible bimodule category which is trivial as a left module category is the trivial bimodule category $\C$. Thus all nontrivial maximal rank invertible module categories are not of the form $\operatorname{Out}(\C)$ and thus provide interesting examples.

We will now provide a characterization of maximal rank invertible module categories for non-degenerate fusion categories in terms of braided autoequivalences. In \cite{D}, Davydov introduced the notion of a \textit{soft} monoidal functor, which is simply a monoidal functor which is isomorphic to the identity functor as a linear functor. Equivalently, a soft monoidal functor is one which fixes equivalence classes of objects.

Recall from \cite{ENO2},\cite[Section 2.9]{DN1}, $\alpha$-induction provides us with an isomorphism $\partial:\operatorname{Pic}(\C)\rightarrow \operatorname{Aut}^{br}(\C)$. The following result is originally due to Kirillov Jr \cite{Kir} (see also \cite[Section II.3]{HQFT}) in the case of modular categories.

\begin{proposition}
If $\C$ is a non-degenerate braided fusion category and $\M$ is an invertible module category, the rank of $\M$ is the number of equivalence classes of simple objects fixed by $\partial(\M)$. In particular, the image of the group of maximal rank invertible module categories is the group of soft braided tensor autoequivalences of $\C$.
\end{proposition} 

\begin{proof}
$\M$ induces a braided autoequivalence of $\Psi(\M)\in \mathcal{Z}(\C)$, which by \cite[Lemma 4.4]{DN1} is $\operatorname{Id}_{\C}\boxtimes \partial$, acting on $\mathcal{Z}(\C)\cong \C\boxtimes \C^{\rev}$. But 

$$I(\mathbb{1})\cong \bigoplus_{X\in \text{Irr}(\C)} X\boxtimes X^{*}$$ hence $$\Psi(\M)(I(\mathbb{1}))=\bigoplus_{X\in \text{Irr}(\C)} X\boxtimes \partial(\M)(X^{*}).$$ 

\noindent Thus $\operatorname{rank}(\M)=\dim(\operatorname{Hom}_{\C\boxtimes \C^{\rev}}( I(\mathbb{1}),\Psi(\M)(I(\mathbb{1})) ) )$ is precisely the number of fixed points of $\partial(\M)$ acting on $\text{Irr}(\C)$.
\end{proof}

Davydov \cite{D} has computed the group of soft braided autoequivalences for the non-degenerate braided tensor category $\mathcal{Z}(\operatorname{Vec}(G))$ for finite groups $G$. The answer is somewhat involved, but he shows it is a certain subgroup of the image of $\operatorname{Out}(\operatorname{Vec}(G))\cong H^{2}(G, \mathbb{C}^{\times})\rtimes\operatorname{Out}(G)$ inside $\operatorname{Aut}^{br}(\mathcal{Z}(\Vec(G)))$ satisfying a compatibility condition with respect to double class functions \cite{D}, Theorem 2.12. He then presents several examples which have non-trivial soft braided autoequivalences, the smallest of which has order 64, though there may certainly be smaller examples. In any case, these provide examples of non-trivial maximal rank invertible module categories.

\section{Rank finiteness for braided fusion categories}

The rank finiteness theorem for modular categories was proved in \cite{BNRZ}.
It states that up to a braided equivalence there exists only finitely many modular categories
of any given rank. Below we extend this result  to braided fusion categories that are not
necessarily spherical or non-degenerate. The plan is first to establish this result for non-degenerate
and slightly degenerate categories and then  pass to equivariantizations.

\begin{corollary}\label{graded rank}
Let $\C =\oplus_{a\in A}\, \C_a$ be a fusion category faithfully graded by 
a group $A$. Then $\rank(\C) \leq |A| \rank(\C_e)$.
\end{corollary}
\begin{proof}
The components $\C_a$ are invertible $\C_e$-bimodule categories so this is immediate from Proposition \ref{rank bound}.
\end{proof}

\begin{lemma}
\label{equivariant rank}
Let $\C$ be a fusion category and let $G$ be a finite group acting on $G$. Then 
\[
\frac{1}{|G|} \rank (\C) \leq \rank(\C^G) \leq |G| \rank(\C).
\]
\end{lemma}
\begin{proof}
Simple objects of $\C^G$ are parameterized by pairs consisting of orbits of simple objects of $\C$ under the action of $G$ 
and certain irreducible projective representations of stabilizers. Each orbit has at  most $|G|$ elements,
so the number of orbits is at least $\rank(\C)/|G|$. This implies the first inequality.

On the other hand, there are at most $\rank(\C)$ orbits and each stabilizer has at most $|G|$ irreducible
projective representations, which gives the second inequality.
\end{proof}

\begin{proposition}
\label{nd fin}
There are finitely many equivalence classes of non-degenerate braided fusion categories of any given rank.
\end{proposition}
\begin{proof}
Let $N$ be a positive integer. By \cite{BNRZ}, it suffices to show that there is a positive integer $M$ such that any non-degenerate braided fusion category $\C$ 
of rank $N$ is a subquotient of  a modular  category of rank $\leq M$.
Here by a subquotient we mean a surjective image of a subcategory.
Let $\tilde{\C}$ be the sphericalization of $\C$ \cite{ENO1}. It is a degenerate ribbon category
(its symmetric center is $\Rep(\mathbb{Z}/2\mathbb{Z})$ with  a non-unitary ribbon structure) of rank $2N$. 

As $\tilde{\C}$ is a $\mathbb{Z}/2\mathbb{Z}$-equivariantization
of $\C$,
its center $\Z(\tilde{\C})$ is a 
$\mathbb{Z}/2\mathbb{Z}$-graded modular category with the trivial component $\Z(\tilde{\C})_0= \Z(\C)^{\mathbb{Z}/2\mathbb{Z}}$ by \cite{GNN}.
Using Corollary~\ref{graded rank} and Lemma \ref{equivariant rank} we estimate
\[
\rank(\Z(\tilde{\C})) \leq 2\, \rank (\Z(\tilde{\C})_0) = 2\, \rank (\Z({\C})^{\mathbb{Z}/2\mathbb{Z}}) \leq 4\, \rank( \Z({\C}))  = 4N^2,
\]
so one can take $M = 4N^2$. Indeed, $\C$ is a quotient of $\tilde{\C}$ and so is a subquotient of $\Z(\tilde{\C})$. 
\end{proof}

Let $\C_1,\,\C_2$ be braided fusion categories with embeddings $\sVec \hookrightarrow \Z_{sym}(\C_i),\, i=1,2$.
Then $\C_1 \bt_{\sVec} \C_2$  has a canonical structure of a braided fusion category \cite{DNO}. 
Namely, it is equivalent to the category of $A$-modules in
$\B_1 \bt \B_2$, where $A$ is the regular algebra of the maximal Tannakian subcategory of $\sVec \bt \sVec \subset \C_1 \bt \C_2$.
If $\C_1$ and $\C_2$ are slightly degenerate then so is $\C_1 \bt_{\sVec} \C_2$.

\begin{proposition}
\label{sd fin}
There are finitely many equivalence classes of  slightly degenerate braided fusion categories of any given rank.
\end{proposition}
\begin{proof}
Let $\C$ be a slightly degenerate braided fusion category of rank $N$. Its center $\Z(\C)$ contains
a fusion subcategory $\C \vee \C^\rev \cong \C \bt_{\sVec}\C^\rev$ of Frobenius-Perron dimension $\frac{1}{2} \FPdim(\C)^2 = \frac{1}{2} \FPdim(\Z(\C))$.
Hence, $\Z(\C)$  is $\mathbb{Z}/2\mathbb{Z}$-graded
by Lemma~\ref{Z2 grad} and
\[
\rank(\Z(\C)) \leq  2\, \rank (\C \bt_{\sVec}\C^\rev) = 2 \times \frac{N^2}{2}  = N^2
\]
by Corollary~\ref{graded rank}. Since $\C$ is a fusion subcategory of $\Z(\C)$ the result follows.
\end{proof}

\begin{remark}
It was observed in \cite{BGNPRW}, following \cite{BRWZ} that if $\C\subset \tilde{\C}$ is a minimal modular extension of a super-modular category then $\frac{3}{2}\rank(\C)\leq \rank(\tilde{\C})\leq 2\rank(\C)$.  This could be used in place of the more general Corollary~\ref{graded rank} in the proof above. 
\end{remark}

\begin{theorem}\label{rf for bfcs}
There  are finitely many equivalence classes of  braided fusion categories of any given rank.
\end{theorem}
\begin{proof}
Let $\C$ be a braided fusion category of rank $N$.  Let $\E \cong \Rep(G)$ be the maximal
Tannakian subcategory of $\Z_{sym}(\C)$. Then $\C = \D^G$, where $\D$ is either non-degenerate
or slightly degenerate braided fusion category. By Lemma~\ref{equivariant rank}
\[
\rank(\D) \leq |G| \rank(\C) = |G| N.
\]
Now let $M$ be the maximal order of a group with at most $N$ isomorphism classes of irreducible representations
($M$ exists since the number of such groups is finite by Landau's theorem).  We have $\rank(\D) \leq MN$, so there
are finitely many choices for $\D$, thanks to Lemmas~\ref{nd fin} and ~\ref{sd fin}. There are also finitely many
choices for the group $G$ and for each such a choice there are finitely many  different
actions of $G$ on $\D$ \cite{ENO1}. 
Thus, there are finitely many possible $\C$'s. 
\end{proof}

\begin{corollary}
There  are finitely many equivalence classes of $G$-crossed braided fusion categories of any given rank.
\end{corollary}
\begin{proof}
Follows immediately from  Theorem~\ref{rf for bfcs} and 
Lemma~\ref{equivariant rank}, since any $G$-crossed braided fusion category is obtained as a de-equivariantization of a braided fusion category \cite[Theorem 4.4.]{DGNO}.
\end{proof}

\begin{section}{The center of a slightly degenerate braided fusion category}

Let $\C$ be a slightly degenerate braided fusion category.  
We have  $\Z_{sym}(\C) \cong \sVec$.  Let $\delta$ 
denote the non-trivial invertible object in $\Z_{sym}(\C)$.

For any $\C$-module category $\M$ let us denote
\[
\M^s :=  \M \bt_{\sVec} \Vec.
\]
In particular, $\C^s := \C \bt_{\sVec} \Vec$
is equivalent to the category of $A$-modules in $\C$, where
$A$ is the regular algebra of $\sVec$. 
We have $\M^s= \M\bt_\C \C^s$. 
Note that $\text{rank}(\C^s) =\frac{1}{2}\rank(\C)$.

\begin{lemma}
$\C^s$ is an invertible $\C$-module category of order $2$.
\end{lemma}
\begin{proof}
This follows from  straightforward equivalences:
\[
\C^s \bt_\C \C^s  = (\C \bt_{\sVec} \Vec) \bt_\C (\C \bt_{\sVec} \Vec)
\cong \C \bt_{\sVec}  (\Vec \bt_{\sVec} \Vec) \cong \C,
\]
where we used the obvious fact $\Vec \bt_{\sVec} \Vec \cong \sVec$.
\end{proof}

\begin{lemma}
We have  $\C^s\bt_\C \M \cong \M \bt_\C \C^s$ for any $\C$-module category $\M$.
\end{lemma}
\begin{proof}
Let $B\in \C$ be an algebra such that $\M\cong \C_B$.  Then $A\ot B \cong B\ot A$
as algebras since $A\in \Z_{sym}(\C)$.  This yields the statement.
\end{proof}

Let $\C_1,\, \C_2$ be slightly degenerate braided fusion categories.
Let 
\[
E \in  \sVec \bt \sVec  \subset \C_1\bt \C_2
\] 
be a canonical \'etale algebra.
Recall that the braided fusion category $\C_1\bt_{\sVec} \C_2$ is defined as
the category of $E$-modules in $\C_1\bt \C_2$.  There are obvious embeddings
$\C_1,\, \C_2 \hookrightarrow \C_1\bt_{\sVec} \C_2$.

Let $\M_1$ and $\M_2$ be module categories over $\C_1$ and $\C_2$.
Define a $\C_1\bt_{\sVec} \C_2$-module category  
$\M_1\bt_{\sVec} \M_2$ to be the category of $E$-modules in $\M_1\bt \M_2$ with the module action given by
\[
X \odot M  = X \ot_E M,\qquad X \in \C_1\bt_{\sVec} \C_2,\, M \in \M_1\bt_{\sVec} \M_2.
\]

Let $\M$ be an indecomposable  $\C_1\bt_{\sVec} \C_2$-module category and let
\begin{equation*}
\M =\bigoplus_{i\in I}\, \M_i,\qquad \M = \bigoplus_{j\in J}\, \N_j
\end{equation*}
be its decompositions into direct sums of indecomposable $\C_1$-module categories
and $\C_2$-module categories, respectively.

\begin{proposition}
\label{when product}
There exist indecomposable $\C_i$-module categories $\L_i,\,i=1,2,$ such that $\M \cong \L_1\bt_{\sVec} \L_2$
if and only if  $\M_i\cap \N_j$ is an  indecomposable $\sVec$-module category for some $i\in I$ and $j\in J$.
\end{proposition}
\begin{proof}
One implication is obvious. 

Suppose that $\M_i\cap \N_j$ is an  indecomposable $\sVec$-module category.  There are two possible
cases.

(Case 1) $\M_i\cap \N_j \cong \sVec$.  Let $X\in \M_i\cap \N_j$ be a simple object. 
Let $\delta_i$ denote the non-trivial invertible object in $\C_i,\, i=1,2$.  Then $\delta_i\ot X \not\cong X$.  
Let us  view $\M$ as a $\C_1\bt  \C_2$-module category and
compute the internal Hom:
\begin{eqnarray*}
\lefteqn{\uHom_{\C_1\bt  \C_2}(X,\, X)} \\
 &\cong& \uHom_{\C_1}(X,\, X) \bt \uHom_{\C_2}(X,\, X)
\oplus  \uHom_{\C_1}(X,\, \delta_1\ot  X) \bt \uHom_{\C_2}(\delta_2 \ot X,\, X) \\
&\cong& \left( \uHom_{\C_1}(X,\, X) \bt \uHom_{\C_2}(X,\, X) ) \right) \ot E,
\end{eqnarray*}
where $E = \be \ot \be \oplus \delta_1\bt \delta_2$ is the canonical algebra in $\sVec\bt \sVec
\subset \C_1\bt \C_2$.  Therefore, as a $\C_1\bt_{\sVec} \C_2$-module category,  $\M \cong \L_1\bt_{\sVec} \L_2$, where
$\L_i$ is the category of $ \uHom_{\C_i}(X,\, X)$-modules in $\C_i,\,i=1,2$.

(Case 2) $\M_i\cap \N_j \cong \Vec$. In this case the  $\C_1\bt_{\sVec} \C_2$-module category
\[
(\C_1^s \bt_{\sVec} \C_2) \bt_{\C_1\bt_{\sVec} \C_2} \M
\]
satisfies the condition of (Case 1) above and, hence,  is equivalent to $\L_1\bt_{\sVec} \L_2$.
Consequently,  $\M \cong \L_1^s\bt_{\sVec} \L_2$.
\end{proof}

\begin{remark}
The pair  of module categories $\L_1,\, \L_2$ in Proposition~\ref{when product}
is determined up to a simultaneous substitution of $\L_1,\, \L_2$ by
$\L_1^s,\, \L_2^s$. 
\end{remark}

\begin{example}
Let $\C_1=\C_2=\sVec$.  Then $\C_1\bt_{\sVec} \C_2 =\sVec$ and
\begin{eqnarray*}
\sVec &\cong&  \sVec\bt_{\sVec} \sVec \cong \Vec\bt_{\sVec} \Vec, \\
\Vec &\cong&  \Vec\bt_{\sVec} \sVec \cong \sVec\bt_{\sVec} \Vec
\end{eqnarray*}
as $\sVec$-module categories.
\end{example}

\begin{proposition}
Let $\C$ be a slightly degenerate braided fusion category and let
$\D=\D_0\oplus \D_1$  be a minimal extension  (see Definition~\ref{def min ext}) of  $\D_0:=\C \bt_{\sVec} \C^\rev$. 
There exists an  invertible   $\C$-module (respectively, $\C^\rev$-module) category $\M$
(respectively, $\N$) such that  $\D_1 \cong \M \bt_{\sVec} \N$ as a $\C \bt_{\sVec} \C^\rev$-module category. 

The equivalence classes of module categories $\M$ and $\N$ are determined up to a simultaneous substitution by
$\M^s$ and $\N^s$. 
\end{proposition}
\begin{proof}
Note that $\D$ is a $\mathbb{Z}/2\mathbb{Z}$-graded extension of $\D_0$ by Lemma~\ref{Z2 grad}.

Let $n$ be the number of $\C$-module
components of $\D_1$.  By \cite[Corollary 3.6]{DGNO} the number of $\C$-module
components of $\D$ is equal to the rank of the centralizer of $\C$ in $\D$. The latter is  $\C^\rev$.
Since the number of $\C$-module components in $\D_0 =\C \bt_{\sVec} \C^\rev$ is  $\frac{1}{2}\text{rank}(\C)$ we conclude
that 
\[
n = \frac{1}{2}\text{rank}(\C).
\] 
Note that $n$ is also equal to the number of $\C^\rev$-module
components of $\D_1$.

Let $\oplus_{i=1}^n\, \M_i$ (respectively, $\oplus_{j=1}^n\,\N_j$)
be decompositions of $\D_1$ into direct sums of indecomposable
$\C$-module (respectively, $\C^\rev$-module) subcategories.
In view of Proposition~\ref{when product} it suffices to check
that for some $i,\,j$ the intersection $\M_i\cap\N_j$ is an indecomposable $\sVec$-module category. 

By Proposition~\ref{rank bound} we have
\[
\text{rank}(\D_1) \leq  \text{rank}(\D_0) = \frac{1}{2}\text{rank}(\C)^2=2n^2.  
\]
Since $\D_1$ is indecomposable as a $\D_0$-bimodule category
each  $\M_i\cap \N_j,\, i,j=1,\dots,n$ is non-zero. 
If any of these intersections has rank $1$,  then it is $\sVec$-indecomposable. This happens automatically if either $\text{rank}(\M_i)$ or $\text{rank}(\N_j)$ is less than  $2n$ for some $i$ or $j$ (indeed, $\text{Irr}(\M_i)$ intersects non-trivially
with $n$ disjoint sets $\text{Irr}(\N_j),\, j=1,\dots,n$).

So let us assume that all intersections $\M_i\cap \N_j$ have rank 
$\geq 2$ and that all $\M_i$ and $\N_j$ have rank $\geq 2n$. The latter implies that $\text{rank}(\M_i)=\text{rank}(\N_j) =2n$
and $\text{rank}(\M_i \cap \N_j)=2$
for all $i$ and $j$ since otherwise $\text{rank}(\D_1)>2n\times n= 2n^2$. Hence, $\text{rank}(\D_1) =2n^2 = \text{rank}(\D_0)$, i.e.,
$\D_1$ is a maximal rank invertible $\D_1$-bimodule category.

By Proposition~\ref{maxranksubgroup} the Lagrangian algebras
corresponding to $\D_0$-bimodule categories $\D_0$ and $\D_1$
are isomorphic as objects of $\Z(\D_0)$. In particular, their 
forgetful images in $\D_0$ are isomorphic: 
\[
\bigoplus_{X\in \text{Irr}(\D_0)}X\ot X^* \cong
\bigoplus_{X\in \text{Irr}(\D_1)}X\ot X^*.
\]
The object on the left does not contain $\delta$ since $\delta$   acts freely on $\text{Irr}(\D_0)$ by \cite[Lemma 3.28]{DGNO}. 
Hence, the same is true for the object on the right, i.e., $\delta$  also acts freely on  
$\text{Irr}(\D_1)$. Thus, every $\M_i\cap \N_j$ is $\sVec$-indecomposable and
$\D_1= \M \bt_{\sVec} \N$ by Proposition~\ref{when product}.


The following equivalences:
\begin{eqnarray*}
\D_0 &\cong& \D_1\bt_{\D_0} \D_1 \\
&\cong&  ( \M \bt_{\sVec} \N) \bt_{\C \bt_{\sVec} \C^\rev} ( \M \bt_{\sVec} \N) \\
&\cong& (\M \bt_\C \M) \bt_{\sVec} (\N \bt_{\C^\rev} \N),
\end{eqnarray*}
imply  that  $\M \bt_\C \M$ is equivalent to $\C$ or $\C^s$
and, hence, $\M$ is invertible. Similarly, $\N$ is invertible.

\end{proof}

\begin{corollary}\label{center corollary}
Let $\C$ be a slightly degenerate braided fusion category. There exist an invertible $\C$-module categories $\M$ and $\N$  such that 
\[
\Z(\C) \cong (\C \bt_{\sVec} \C^\rev) \oplus (\M \bt_{\sVec} \N)
\]
as a $\C \bt_{\sVec} \C^\rev$-module category. 
\end{corollary}

\begin{remark}
It is possible to show that the above $\M$ and $\N$ are {\em braided} $\C$-module
categories of order $2$, see \cite{DN2}.
\end{remark}

\begin{remark}
It will be interesting to see if $\C$ is a slightly degenerate braided fusion category then for such an $\mM$ as above $\C\oplus \mM$ 
has a structure of a minimal extension of $\C$.  One expects that there are $16$ choices of $\mM$ in this case, by the results of \cite{BGNPRW,KLW}.
Notice that if $\tilde{\C}=\C\oplus\mN$ \textit{is} a minimal extension of some slightly degenerate braided fusion category $\C$ then $\Z(\C)$ has the form as in Corollary \ref{center corollary}, as can be seen as follows: $\Z(\tilde{\C})\cong \tilde{\C}\boxtimes\tilde{\C}^\rev$ contains a Tannakian subcategory $\mD\cong\Rep(\mathbb{Z}/2\mathbb{Z})$ as the diagonal of $\sVec\boxtimes\sVec$. The centralizer of $\mD$ in $\Z(\tilde{\C})$ is $(\C \bt \C^\rev) \oplus (\N \bt\N^{\rev})$, so that the de-equivariantization is $$(\C \bt_{\sVec} \C^\rev) \oplus (\N \bt_{\sVec} \N^{\rev})\cong[\Z(\tilde{\C})_{\mathbb{Z}/2\mathbb{Z}}]_0\cong \Z(\C).$$
\end{remark}
\end{section}

\bibliographystyle{ams-alpha}

\end{document}